\documentclass[12pt]{amsart}

\usepackage{amsmath}
\usepackage{amssymb}
\usepackage{latexsym}
\usepackage{graphicx} 
\usepackage{mathrsfs}
\usepackage{mathtools}
\usepackage{amsthm}
\usepackage{tikz}
\usetikzlibrary{cd}
\allowdisplaybreaks[4]

\usepackage[colorlinks,citecolor=blue,linkcolor=blue,linktocpage,unicode]{hyperref}

\numberwithin{equation}{section}
\numberwithin{figure}{section}

\newtheorem{thm}{Theorem}[section]
\newtheorem{prop}[thm]{Proposition}
\newtheorem{lemma}[thm]{Lemma}
\newtheorem{cor}[thm]{Corollary}

\newtheorem{conj}[thm]{Conjecture}

\theoremstyle{definition}
\newtheorem{example}[thm]{Example}
\newtheorem{defin}[thm]{Definition}
\newtheorem{remark}[thm]{Remark}

\DeclareMathOperator{\rad}{rad}
\DeclareMathOperator{\good}{Good}
\DeclareMathOperator{\id}{id}

\title[Decompositions of good involutions]{Decompositions of good involutions on quandles}
\date{\today}
\subjclass[2020]{57K12, 20N02, 13A99}
\keywords{quandle, good involution, symmetric quandle, symplectic quandle, interaction-free union of quandles, direct product of quandles}

\author{Yasuhito Nakajima}
\address{Yasuhito Nakajima, Independent Researcher}
\email{yasuhito.nakajima.mathematics@gmail.com}

\author{Kentaro Yamaguchi}
\address{Kentaro Yamaguchi, Research Institute for Mathematical Sciences, Kyoto University, Kyoto 606-8502, Japan}
\email{c201596e@alm.icu.ac.jp}
\begin{document}
\begin{abstract}
We study the behavior of good involutions under two fundamental constructions of quandles: interaction-free unions and direct products.
In particular, we show that the set of good involutions on a quandle decomposes naturally into the set of involutions on its maximal trivial component and the set of good involutions on the complement.
Moreover, we construct an example of a connected noninvolutory symmetric quandle with multiple good involutions, which serves as a counterexample to a conjecture by Ta. 
\end{abstract}

\maketitle

\tableofcontents

%%%%%%%%%%%%%%%%%%%%%%%%%%%%%%%%%%%%%%%
\section{Introduction}
\label{sec: introduction}

A \emph{quandle} is an algebraic structure introduced independently by Joyce \cite{MR638121} and Matveev \cite{MR672410} in knot theory.
Quandles serve as a powerful tool to construct invariants of (oriented) knots.
From an algebraic perspective, a quandle is an idempotent right-invertible self-distributive algebraic structure. This structure appears in various areas of mathematics, for instance, as an axiomatization of the properties of conjugation in a group or as a generalization of point symmetries of symmetric spaces.

%%%
Beyond their applications to low-dimensional topology, quandles have also been studied as algebraic structures in their own right.
A central topic in the structure theory of algebraic systems is to understand how algebraic properties behave under basic operations, such as direct products, disjoint unions, and extensions.
In quandle theory, Cartesian products and disjoint unions are fundamental constructions used to produce new quandles from a family of quandles, such as a direct product of quandles, and an interaction-free union of quandles (cf. \cite{MR4410185}).
$G$-families of quandles and $Q$-families of quandles \cite{MR2443248,MR3275740,MR3430992} provide other important examples of such quandles constructed from a family of quandles.
%%%

A \emph{symmetric quandle} \cite{MR2657689} is a quandle equipped with a \emph{good involution} \cite{MR2371714}.
Symmetric quandles are useful to study knots that are not necessarily oriented or orientable.
The existence of good involutions on a given quandle (and hence the existence of a symmetric quandle structure) is quite nontrivial; thus Taniguchi posed the problems of their existence and classification in \cite{MR4688855}.
Several classification results have appeared in the literature.
For instance, Kamada and Oshiro \cite{MR2657689} classified good involutions on involutory quandles, trivial quandles, and dihedral quandles.
Taniguchi \cite{MR4688855} studied the existence of good involutions on generalized Alexander quandles and provided their classification.
More recently, Ta investigated the existence of good involutions on conjugation subquandles \cite{ta2025goodinvolutionsconjugationsubquandles}, as well as on twisted conjugation subquandles and Alexander quandles \cite{ta2025goodinvolutionstwistedconjugation}.
In our previous work \cite{NY2604}, we established the nonexistence of good involutions on symplectic quandles.

%%%
From an algebraic perspective, a natural question is how good involutions behave under these constructions.
In particular, we investigate when the set of good involutions on a product or interaction-free union of quandles is completely decomposed into the product of those on its components.
%%%

%%%%%%%%%%%%%%%%%%%%%%%%%%%%%%%%%%%%%%%
\subsection{Main results}
In this paper, we study how good involutions behave under two fundamental constructions of quandles: interaction-free unions and direct products.

First, we present our main results regarding the interaction-free union of quandles.
For any quandle $(Q,\ast)$, let $D$ denote the \emph{maximal trivial component} of $Q$. 
Note that $(D,\ast|_{D})$ is a trivial quandle.
Then every quandle $Q$ decomposes as an interaction-free union of the maximal trivial component $D$ and its complement $Q\setminus D$.
In this setting, we obtain the following result:
\begin{thm}\label{main 0: interaction-free union}
Let $Q$ be a quandle and $D$ denote the maximal trivial component of $Q$. If $D \neq \emptyset$ and $Q\setminus D \neq \emptyset$, then there is a canonical bijection
\begin{equation*}
    \good (Q,\ast)
    \cong 
    \good(Q\setminus D, \ast)\times \operatorname{Inv}(D).
\end{equation*}
\end{thm}

A complete decomposition of good involutions does not hold without additional assumptions on their components; for instance, the set of good involutions on an interaction-free union of trivial quandles fails to decompose into the product of those on each component (see Example~\ref{eg: family of trivial qdle}).
To overcome this obstruction, we impose the connectedness of the components.
Under this natural condition, we establish the following result (Theorem~\ref{main thm: decomp of good inv on interaction-free}):

\begin{thm}\label{main 1: interaction-free union}
Let $\{(Q_{\lambda},\ast_{\lambda})\}_{\lambda \in \Lambda}$ be a family of nontrivial connected quandles and $\bigsqcup^{\mathrm{free}}_{\lambda \in \Lambda} Q_{\lambda}$ denote their interaction-free union.
Then there is a canonical bijection
\begin{equation*}
    \good(\bigsqcup^{\mathrm{free}}_{\lambda \in \Lambda} Q_{\lambda})
    \cong 
    \prod_{\lambda \in \Lambda} \good(Q_{\lambda},\ast_{\lambda}).
\end{equation*}
\end{thm}

As an application of the decomposition theorem (Theorem~\ref{main 0: interaction-free union}), we generalize our previous result \cite[Theorem 5.3]{NY2604} regarding the set of good involutions on symplectic quandles.
More precisely, a generalized symplectic quandle is defined on a module $M$ over a commutative ring $R$ with identity, equipped with an antisymmetric bilinear form $\langle, \rangle \colon M \times M \to R$ (see Definition~\ref{def: generalized symplectic quandles}).
If $M$ is a free module of finite rank over $R$, then the generalized symplectic quandle reduces to a classical symplectic quandle, which was introduced by Navas and Nelson \cite{MR2493966}.
Note that if the characteristic of $R$ is $2$, then symplectic quandles over $R$ are involutory and the identity map serves as a good involution on any involutory quandle. 

In fact, for any generalized symplectic quandle $(M,\ast)$, the radical $\rad(M)$ of the underlying module $M$ coincides with the maximal trivial component of $(M,\ast)$.
Thus, every generalized symplectic quandle $(M,\ast)$ decomposes as an interaction-free union of its subquandles $(M \setminus \rad(M),\ast)$ and $(\rad(M),\ast)$.
Note that since the subquandle $(\rad(M),\ast)$ itself is a trivial quandle, it follows that $\good(\rad(M),\ast) = \operatorname{Inv}(\rad(M))$.
Consequently, we obtain the following (Theorem~\ref{main thm: symplectic qdle}):
\begin{thm}\label{main 2: symplectic quandles}
Let $(M,\ast)$ be a generalized symplectic quandle over an integral domain $R$. If $M \neq \rad(M)$, then there is a canonical bijection
\begin{equation*}
    \good(M,\ast) \cong \good(M\setminus \rad(M),\ast) \times \operatorname{Inv}(\rad(M)).
\end{equation*}
\end{thm}
Moreover, if the characteristic of $R$ is $2$ and the module $M$ is torsion-free, then we have a bijection
\begin{equation*}
\good(M,\ast) \cong \operatorname{Inv}(\rad(M))
\end{equation*} 
(see Corollary~\ref{cor: char2 implies good is inv of rad}).
Thus, the above theorem implies that there exists no good involution arising from the structure of torsion-free generalized symplectic quandles over an integral domain of characteristic $2$.
%%%

Next, we turn to our main result regarding the product of quandles.
As in the case of interaction-free unions, a complete decomposition of good involutions does not hold without additional assumptions on their components (see Example~\ref{eg: family of trivial qdle, product}).
To overcome this obstruction, we impose the connectedness of the components and finiteness of the family.
Under these conditions, we establish the following result (Theorem~\ref{main thm: good involutions on direct products}):
\begin{thm}\label{main 3: direct product}
Let $\{(Q_{\lambda},\ast_{\lambda})\}_{\lambda = 1,\ldots,n}$ be a finite family of nontrivial connected quandles
and $(\prod_{\lambda = 1}^{n}Q_{\lambda},\ast)$ denote their direct product.
Then there is a canonical bijection
\begin{equation*}
    \good(\prod_{\lambda = 1}^{n}Q_{\lambda})
    \cong 
    \prod_{\lambda = 1}^{n} \good(Q_{\lambda},\ast_{\lambda}).
\end{equation*}
\end{thm}

\begin{cor}
Let $\{(Q_{\lambda},\ast_{\lambda})\}_{\lambda = 1,\ldots,n}$ be a finite family of nontrivial connected quandles.
Let $(\bigsqcup^{\mathrm{free}}_{\lambda =1,\ldots,n}Q_{\lambda},\ast_{\mathrm{free}})$ denote their interaction-free union 
and $(\prod_{\lambda = 1}^{n}Q_{\lambda},\ast)$ denote their direct product.
Then there is a canonical bijection
\begin{equation*}
    \good(\bigsqcup^{\mathrm{free}}_{\lambda = 1,\ldots,n} Q_{\lambda})
    \cong 
    \good(\prod_{\lambda = 1}^{n}Q_{\lambda})
    \cong 
    \prod_{\lambda = 1}^{n} \good(Q_{\lambda},\ast_{\lambda}).
\end{equation*}
\end{cor}

Direct products of quandles preserve the connectedness of their direct factors, whereas they do not preserve the involutory property in general.
Exploiting this feature, as an application of Theorem~\ref{main 3: direct product}, we construct a connected noninvolutory quandle equipped with multiple good involutions, by taking the direct product of an involutory connected quandle and a connected noninvolutory quandle that admits a good involution (see Example~\ref{eg: counterexample to conjecture by Ta}).
Consequently, this provides a counterexample to the following conjecture proposed by Ta:
\begin{conj}[{\cite[Conjecture 11.8]{ta2025goodinvolutionsconjugationsubquandles}}]
Let $(Q,\ast)$ be a connected noninvolutory quandle.
Then we have $\lvert \good(Q,\ast) \rvert \leq 1$.
\end{conj}

%%%%%%%%%%%%%%%%%%%%%%%%%%%%%%%%%%%%%%%
\subsection{Outline}
This paper is organized as follows.
In Section~\ref{sec: quandle and good involutions}, we recall the notions of quandles and good involutions and review some of their properties.
In Section~\ref{sec: good involutions on interaction-free union}, we investigate how the set of good involutions on an interaction-free union of quandles decomposes into the Cartesian product of those on its components.
In Section~\ref{sec: symplectic quandle},
we introduce the notion of generalized symplectic quandles.
In Section~\ref{sec: good involutions on symplectic quandles},
we apply the decomposition result of interaction-free unions of quandles to the set of good involutions on generalized symplectic quandles.
In Section~\ref{sec: good involutions on product}, 
we investigate how the set of good involutions on a direct product of quandles decomposes into the Cartesian product of those on a direct factor and provide an example of a connected noninvolutory symmetric quandle equipped with multiple good involutions.

%%%%%%%%%%%%%%%%%%%%%%%%%%%%%%%%%%%%%%%
\section{Quandles and good involutions} \label{sec: quandle and good involutions}

\subsection{Quandles}

\begin{defin}[{\cite{MR638121,MR672410}}]\label{def: quandle}
A \emph{quandle} is a set $Q$ equipped with a binary operation $\ast \colon Q \times Q \to Q$ satisfying:
\begin{enumerate}
    \item $x\ast x = x$ for all $x \in Q$,
    \item the map $s_{y}\colon Q \to Q$ defined by $x \mapsto x \ast y$ is invertible for all $y \in Q$,
    \item $(x \ast y) \ast z = (x \ast z) \ast (y \ast z)$ for all $x,y,z \in Q$. 
\end{enumerate}
\end{defin}

The quandle axioms correspond to the Reidemeister moves from knot theory.

\begin{remark} \label{remark: dual quandle}
Let $(Q,\ast)$ be a quandle.
The second axiom implies that there exists a \emph{dual} (or \emph{inverse}) operation $\ast^{-1} \colon Q \times Q \to Q$ satisfying for all $x,y \in Q$, $x \ast^{-1}y \coloneqq s_{y}^{-1}(x)$.
Moreover, the pair $(Q,\ast^{-1})$ is also a quandle, which is called the \emph{dual quandle} of $(Q,\ast)$. 
\end{remark}

From an algebraic perspective, quandles can be viewed as a generalization of the conjugation operation in a group.
Consequently, several notions appearing in group theory can be defined analogously in quandle theory.

\begin{defin}
Let $(Q_{1},\ast_{1}), (Q_{2},\ast_{2})$ be quandles.
\begin{itemize}
    \item A map $f\colon Q_{1} \to Q_{2}$ is a \emph{quandle homomorphism} if it satisfies $f(x \ast_{1} y) = f(x) \ast_{2} f(y)$ for all $x,y \in Q_{1}$.
    \item A map $f\colon Q_{1} \to Q_{2}$ is a \emph{quandle isomorphism} if it is a bijective quandle homomorphism.
        $(Q_{1},\ast_{1})$ is \emph{isomorphic} to $(Q_{2},\ast_{2})$ if there exists a quandle isomorphism from $Q_{1}$ to $Q_{2}$.
\end{itemize}
\end{defin}

\begin{defin}
Let $(Q,\ast)$ be a quandle and $(Q,\ast^{-1})$ the dual quandle.
\begin{itemize}
    \item The set of all quandle isomorphisms from $Q$ to itself is denoted by $\operatorname{Aut}(Q,\ast)$. $\operatorname{Aut}(Q,\ast)$ is a group under the binary operation $fg \coloneqq g \circ f$, which is called the \emph{automorphism group}.
    \item The \emph{inner automorphism group} $\operatorname{Inn}(Q,\ast)$ of $(Q,\ast)$ is a subgroup of $\operatorname{Aut}(Q,\ast)$ generated by $\{s_{y} \mid y \in Q\}$.
    \item A map $f \colon Q \to Q$ is an \emph{antiautomorphism} if it is a quandle homomorphism from $(Q,\ast)$ to $(Q,\ast^{-1})$. $(Q,\ast)$ is \emph{self-dual} if there exists an antiautomorphism of $(Q,\ast)$.
\end{itemize}
\end{defin}

We provide some important classes of quandles.
\begin{defin}
Let $(Q,\ast)$ be a quandle.
\begin{itemize}
    \item $(Q,\ast)$ is \emph{involutory} or a \emph{kei} if $\ast = \ast^{-1}$ (cf. \cite{MR21002}).
    \item $(Q,\ast)$ is \emph{homogeneous} if the automorphism group $\operatorname{Aut}(Q,\ast)$ acts transitively on $Q$.
    \item $(Q,\ast)$ is \emph{connected} if the inner automorphism group $\operatorname{Inn}(Q,\ast)$ acts transitively on $Q$.
    \item $(Q,\ast)$ is \emph{faithful} if $s_{x} = s_{y}$ implies $x = y$ for every $x,y \in Q$.
\end{itemize}
\end{defin}

The following examples of quandles play an important role in the present paper.

\begin{example}\label{eg: trivial quandle}
Let $Q$ be a nonempty set.
Define the binary operation $\ast \colon Q \times Q \to Q$ by $x \ast y \coloneqq x$. 
The pair $(Q,\ast)$ is a quandle, which is called a \emph{trivial quandle}.
Note that every trivial quandle is involutory.
\end{example}

\begin{defin}[{\cite{MR2493966}}]\label{def: maximal trivial component}
Let $(Q,\ast)$ be a quandle.
Define the set $D$ by 
\begin{equation*}
    D \coloneqq 
    \{ x \in Q \mid x\ast y = x, y \ast x = y \text{ for all }y \in Q\}.
\end{equation*}
Then $D$ is a subquandle of $Q$, which is itself a trivial quandle.
We call $D$ the \emph{maximal trivial component} of $Q$.
\end{defin}

\begin{example}\label{eg: conjugation quandle}
Let $G$ be a group and $X$ a conjugacy class of $G$.
Define the binary operation $\ast \colon X \times X \to X$ by $g\ast h \coloneqq h^{-1}gh$.
The pair $(X,\ast)$ is a quandle, which is called a \emph{conjugation (sub)quandle} and denoted by $\operatorname{Conj}(X)$.
\end{example}

\subsection{Good involutions}

\begin{defin}[{\cite{MR2371714,MR2657689}}]\label{def: good involutions}
Let $(Q,\ast)$ be a quandle. 
An involution $\rho \colon Q \to Q$ is a \emph{good involution} if it satisfies the following conditions:
\begin{enumerate}
    \item $\rho(x\ast y) = \rho(x) \ast y$ for any $x, y \in Q$,
    \item $x \ast \rho(y) = x \ast^{-1} y$ for any $x,y \in Q$.
\end{enumerate}
A \emph{symmetric quandle} is a quandle $(Q,\ast)$ equipped with a good involution $\rho$ on $(Q,\ast)$.
The set of good involutions on $(Q,\ast)$ is denoted by $\good(Q,\ast)$.
\end{defin}

Symmetric quandles are a useful tool to define invariants of knots that may not be oriented or orientable.
The following are some known examples of symmetric quandles.

\begin{example}
    Let $(Q,\ast)$ be a trivial quandle.
    Every involution on $Q$ is a good involution on $(Q,\ast)$, i.e. $\good(Q,\ast) = \operatorname{Inv}(Q)$, where $\operatorname{Inv}(Q)$ denotes the set of all involutions on $Q$. 
\end{example}

\begin{example}
    Let $(Q,\ast)$ be an involutory quandle.
    The identity map on $Q$ is a good involution on $(Q,\ast)$, i.e. $\id_{Q} \in \good(Q,\ast)$. 
\end{example}

\begin{example}[cf.{\cite[Section 4.1]{ta2025goodinvolutionsconjugationsubquandles}}]
    Let $X$ be a conjugacy class of a group $G$. Suppose $X$ is closed under inversion (i.e. $X = X^{-1}$).
    The map $\iota \colon X \to X$ defined by $g \mapsto g^{-1}$ is a good involution on the conjugation quandle $\operatorname{Conj}(X)$.
\end{example}

The existence of good involutions on a given quandle is not trivial.
We list some known results from the literature.

\begin{prop}
Let $(Q,\ast)$ be a quandle.
A good involution on $(Q,\ast)$ is an antiautomorphism of $(Q,\ast)$.
In particular, if $(Q,\ast)$ is not self-dual, then $\good(Q,\ast)= \emptyset$, i.e. $(Q,\ast)$ has no good involution.
\end{prop}

\begin{prop}[{\cite[Proposition 3.1]{MR2657689}}]
Let $(Q,\ast)$ be a quandle.
Every involution on $Q$ is a good involution on $(Q,\ast)$ if and only if $(Q,\ast)$ is a trivial quandle.
\end{prop}

\begin{prop}[{\cite[Proposition 3.4]{MR2657689}}]
Let $(Q,\ast)$ be a quandle.
The following are equivalent:
\begin{enumerate}
    \item $(Q,\ast)$ is involutory.
    \item The identity map on $Q$ is a good involution on $(Q,\ast)$.
    \item $(Q,\ast)$ has a good involution which is a quandle homomorphism.
    \item Any good involution on $(Q,\ast)$ is a quandle homomorphism, and $\good(Q,\ast) \neq \emptyset$.
\end{enumerate}
\end{prop}

\begin{prop}[{\cite[Proposition 5.1]{ta2025goodinvolutionsconjugationsubquandles}}]
Let $(Q,\ast)$ be a quandle.
If $(Q,\ast)$ is faithful, then $\lvert \good(Q,\ast) \rvert \leq 1$.
\end{prop}

%%%%%%%%%%%%%%%%%%%%%%%%%%%%%%%%%%%%%%%
\section{Good involutions on an interaction-free union of quandles} \label{sec: good involutions on interaction-free union}

Given a family of quandles, one can construct a quandle structure on their disjoint union; this construction is called an \emph{interaction-free union of quandles}.
Although this construction might have been known previously, the terminology \emph{interaction-free union} was first introduced by Kubo, Nagashiki, Okuda, and Tamaru \cite{MR4410185}.

In this section, we investigate the relationship between good involutions on an interaction-free union of quandles and those on each component.

\begin{defin}\label{def: interaction-free union}
Let $\{(Q_{\lambda}, \ast_{\lambda})\}_{\lambda \in \Lambda}$ be a family of quandles. 
We define a binary operation $\ast_{\mathrm{free}}$ on the disjoint union $\bigsqcup_{\lambda \in \Lambda}Q_{\lambda}$ by 
\begin{equation*}
    x \ast_{\mathrm{free}} y 
    \coloneqq 
    \begin{cases*}
        x \ast_{\lambda} y & if $x,y \in Q_{\lambda}$ for some $\lambda \in \Lambda$, \\
        x & otherwise.
    \end{cases*}
\end{equation*}
The pair $(\bigsqcup_{\lambda \in \Lambda}Q_{\lambda}, \ast_{\mathrm{free}})$ is a quandle, which is called an \emph{interaction-free union of quandles} and denoted by $\bigsqcup^{\mathrm{free}}_{\lambda \in \Lambda} Q_{\lambda}$.
\end{defin}

We first show that every quandle decomposes as an interaction-free union of its maximal trivial component (see Definition~\ref{def: maximal trivial component}) and its complement.

\begin{lemma}\label{lem: interaction-free of maximal trivial component}
Let $(Q,\ast)$ be a quandle and $D$ denote the maximal trivial component of $Q$.
If $D \neq \emptyset$ and $Q\setminus D \neq \emptyset$, then $(Q,\ast)$ decomposes as an interaction-free union
\begin{equation*}
    (Q,\ast) = (Q\setminus D) \bigsqcup^{\mathrm{free}} D.
\end{equation*}
\end{lemma}
\begin{proof}
First, we note that $Q \setminus D$ forms a subquandle of $Q$.
Indeed, since $y \in Q \setminus D$ implies that $s_{y} \neq \id_{Q}$, it follows that $x, y \in Q \setminus D$ implies that $x \ast y \notin D$ and $x \ast^{-1} y \notin D$.

To show $(Q,\ast) = (Q\setminus D) \bigsqcup^{\mathrm{free}} D$, it suffices to verify that $x \ast y = x \ast_{\mathrm{free}} y$ for all $x, y \in Q$.
We check this in the following disjoint cases:
\begin{enumerate}
    \item If $x, y \in Q \setminus D$, then $x \ast_{\mathrm{free}} y = x \ast y$ by the definition of interaction-free unions.
    \item If $x, y \in D$, then $x \ast_{\mathrm{free}} y = x \ast y = x$ by the definition of $D$.
    \item If $x \in D$ and $y \in Q \setminus D$ or $x \in Q \setminus D$ and $y \in D$, then $x \ast y = x$ by the definition of $D$. Hence $x \ast_{\mathrm{free}} y = x = x \ast y$.
\end{enumerate}
Therefore, $x \ast y = x \ast_{\mathrm{free}} y$ holds for all $x, y \in Q$, which completes the proof.
\end{proof}

We then show the first main result.
\begin{thm}\label{main thm: decomp of maximal trivial component}
Let $(Q,\ast)$ be a quandle and $D$ denote the maximal trivial component of $Q$.
If $D \neq \emptyset$ and $Q\setminus D \neq \emptyset$, then we have the decomposition of good involutions on $(Q,\ast)$:
\begin{equation*}
    \good(Q,\ast) 
    \cong 
    \good(Q\setminus D,\ast) \times \operatorname{Inv}(D).
\end{equation*}
\end{thm}
\begin{proof}
Since $D$ is a trivial quandle, any involution on $D$ is trivially a good involution, i.e. $\operatorname{Good}(D,\ast) = \operatorname{Inv}(D)$.

Fix $\rho \in \good(Q,\ast)$ and $x \in D$.
For every $y \in Q$, the definition of good involutions implies that 
\begin{equation*}
    \rho(x) \ast y = \rho(x\ast y) = \rho(x),
\end{equation*}
and 
\begin{equation*}
    y \ast \rho(x) = y \ast^{-1}x = y,
\end{equation*}
i.e. $\rho(x) \in D$.
Thus, $\rho(D) \subseteq D$.
Since $\rho$ is an involution, by applying $\rho$ to both sides of $\rho(D) \subseteq D$, we have $D = \rho^{2}(D) \subseteq \rho(D) \subseteq D$.
Thus, we obtain $\rho(D) = D$ and $\rho(Q\setminus D) = Q\setminus D$ for all $\rho \in \good(Q,\ast)$.

We may define the map 
\begin{equation*}
    \Phi \colon \good(Q,\ast) \to \good(Q\setminus D,\ast) \times \operatorname{Inv}(D)
\end{equation*} by 
\begin{equation*}
    \Phi(\rho) \coloneqq (\rho|_{Q\setminus D}, \rho|_{D}).
\end{equation*}
The above argument ensures that $\Phi$ is well-defined.
By construction, $\Phi$ is clearly injective.

It remains to show that $\Phi$ is surjective.
For all $\alpha \in \good(Q\setminus D,\ast)$ and all $\beta \in \operatorname{Inv}(D)$, we define the map $\rho \colon Q \to Q$ by
\begin{equation*}
    \rho(x)
    \coloneqq 
    \begin{cases*}
        \alpha(x) & if $x \in Q \setminus D$, \\
        \beta(x) & if $x \in D$.
    \end{cases*}
\end{equation*}
Since $\alpha,\beta$ are involutions, $\rho^{2} = \id_{Q}$, i.e. $\rho$ is an involution.
If $x, y \in Q\setminus D$, then 
\begin{align*}
    &\rho(x) \ast y = \alpha(x) \ast y = \alpha(x\ast y) = \rho(x\ast y), \\
    &x \ast \rho(y) = x \ast \alpha(y) = x \ast^{-1}y.
\end{align*}
If $x \in D$ and $y \in Q$, then 
\begin{align*}
    &\rho(x) \ast y = \beta(x) \ast y = \beta(x) = \beta(x\ast y) = \rho(x\ast y), \\
    &x \ast \rho(y) = x = x \ast^{-1}y.
\end{align*}
If $x \in Q\setminus D$ and $y \in D$, then 
\begin{align*}
    &\rho(x) \ast y = \rho(x) = \rho(x\ast y), \\
    &x \ast \rho(y) = x  = x \ast^{-1}y.
\end{align*}
Thus, $\rho$ is a good involution on $(Q,\ast)$, which proves the surjectivity of $\Phi$.
Consequently, $\Phi$ is the desired bijection.
\end{proof}

Next, we compare the set of good involutions on an interaction-free union of quandles with that on each component.

\begin{prop} \label{prop: inclusion of decomposition interaction-free union}
Let $\{(Q_{\lambda}, \ast_{\lambda})\}_{\lambda \in \Lambda}$ be a given family of quandles 
and $\bigsqcup^{\mathrm{free}}_{\lambda \in \Lambda} Q_{\lambda}$ denote their interaction-free union.
Then there is a canonical inclusion 
\begin{equation*}
    \prod_{\lambda \in \Lambda} \good(Q_{\lambda},\ast_{\lambda})  \hookrightarrow \good(\bigsqcup^{\mathrm{free}}_{\lambda \in \Lambda} Q_{\lambda}).
\end{equation*}
\end{prop}
\begin{proof}
Let $(\rho_{\lambda})_{\lambda} \in \prod_{\lambda \in \Lambda}\good(Q_{\lambda},\ast_{\lambda})$.
Since the quandles $\{Q_{\lambda}\}_{\lambda \in \Lambda}$ are pairwise disjoint, each $x \in \bigsqcup_{\lambda \in \Lambda}Q_{\lambda}$ belongs to $Q_{\lambda}$ for a unique $\lambda \in \Lambda$.
Thus, we can define a well-defined map $\rho \colon \bigsqcup_{\lambda \in \Lambda}Q_{\lambda} \to \bigsqcup_{\lambda \in \Lambda}Q_{\lambda}$ by 
\begin{equation*}
    \rho(x) \coloneqq \rho_{\lambda}(x).
\end{equation*}
We show that $\rho \in \good(\bigsqcup^{\mathrm{free}}_{\lambda \in \Lambda} Q_{\lambda})$.

First, we show that $\rho$ is an involution.
For any $x \in \bigsqcup_{\lambda \in \Lambda}Q_{\lambda}$, there exists $\lambda \in \Lambda$ such that $x \in Q_{\lambda}$.
For such $\lambda$, we have 
\begin{equation*}
    \rho^2(x) = \rho_{\lambda}^{2}(x) = \id_{Q_{\lambda}}(x) = x,
\end{equation*}
i.e. $\rho$ is an involution.

Next, we show that $\rho(x\ast_{\mathrm{free}} y) = \rho(x) \ast_{\mathrm{free}} y$ for all $x,y \in \bigsqcup_{\lambda \in \Lambda}Q_{\lambda}$.
If $x,y \in Q_{\lambda}$, then we have 
\begin{align*}
    \rho(x\ast_{\mathrm{free}} y) 
    &= 
    \rho (x\ast_{\lambda} y) \\
    &=
    \rho_{\lambda} (x\ast_{\lambda} y) \\
    &=
    \rho_{\lambda} (x) \ast_{\lambda} y \\
    &=
    \rho(x) \ast_{\mathrm{free}} y.
\end{align*}
If $x \in Q_{\lambda}$ and $y \in Q_{\mu}$ for $\lambda \neq \mu$, then we have 
\begin{align*}
    \rho(x\ast_{\mathrm{free}} y) 
    &= 
    \rho(x)\\
    &=
    \rho_{\lambda}(x) \\
    &=
    \rho_{\lambda}(x) \ast_{\mathrm{free}} y \\
    &=
    \rho(x) \ast_{\mathrm{free}} y.
\end{align*}
Thus, $\rho(x\ast_{\mathrm{free}} y) = \rho(x) \ast_{\mathrm{free}} y$ holds for all $x,y \in \bigsqcup_{\lambda \in \Lambda}Q_{\lambda}$.

Finally, we show that $x \ast_{\mathrm{free}} \rho(y) = x \ast_{\mathrm{free}}^{-1} y$ for all $x,y \in \bigsqcup_{\lambda \in \Lambda}Q_{\lambda}$.
If $x,y \in Q_{\lambda}$, then we have
\begin{align*}
    x \ast_{\mathrm{free}} \rho(y)
    &=
    x \ast_{\mathrm{free}} \rho_{\lambda}(y) \\
    &=
    x \ast_{\lambda} \rho_{\lambda}(y) \\
    &=
    x \ast_{\lambda}^{-1} y \\
    &=
    x \ast_{\mathrm{free}}^{-1} y.
\end{align*}
If $x \in Q_{\lambda}$ and $y \in Q_{\mu}$ for $\lambda \neq \mu$, then we have 
\begin{align*}
    x \ast_{\mathrm{free}} \rho(y)
    &=
    x \ast_{\mathrm{free}} \rho_{\mu}(y) \\
    &=
    x \\
    &=
    x \ast_{\mathrm{free}}^{-1} y.
\end{align*}
Thus, $x \ast_{\mathrm{free}} \rho(y) = x \ast_{\mathrm{free}}^{-1} y$ holds for all $x,y \in \bigsqcup_{\lambda \in \Lambda}Q_{\lambda}$.

Therefore, the map $\rho$ is a good involution on $\bigsqcup^{\mathrm{free}}_{\lambda \in \Lambda} Q_{\lambda}$.
\end{proof}

The following example shows that the equality does not hold in general.

\begin{example} \label{eg: family of trivial qdle}
Let $A, B$ be nonempty sets and $A_{\mathrm{tri}}, B_{\mathrm{tri}}$ denote the trivial quandles on $A,B$, respectively.

Since the interaction-free union of these quandles satisfies 
\begin{equation*}
    A_{\mathrm{tri}} \bigsqcup^{\mathrm{free}} B_{\mathrm{tri}} \cong (A\sqcup B)_{\mathrm{tri}}
\end{equation*}
we have 
\begin{equation*}
    \good(A_{\mathrm{tri}} \bigsqcup^{\mathrm{free}} B_{\mathrm{tri}})
    =
    \good((A\sqcup B)_{\mathrm{tri}}) 
    =
    \mathrm{Inv}(A \sqcup B).
\end{equation*}
On the other hand, it is well-known that $\mathrm{Inv}(A) \times \mathrm{Inv}(B) \subsetneq \mathrm{Inv}(A \sqcup B)$ for nonempty sets. 
Since $\good(A_{\mathrm{tri}}) = \mathrm{Inv}(A)$ and $\good(B_{\mathrm{tri}}) = \mathrm{Inv}(B)$, we obtain 
\begin{equation*}
    \good(A_{\mathrm{tri}}) \times \good(B_{\mathrm{tri}}) 
    \subsetneq
    \good(A_{\mathrm{tri}} \bigsqcup^{\mathrm{free}} B_{\mathrm{tri}}).
\end{equation*}
\end{example}

Next, we provide a sufficient condition for the equality to hold.

\begin{lemma}\label{lem: qdle operation is not identity implies decomposition}
Let $\{(Q_{\lambda},\ast_{\lambda})\}_{\lambda \in \Lambda}$ be a family of quandles and $\bigsqcup^{\mathrm{free}}_{\lambda \in \Lambda} Q_{\lambda}$ denote their interaction-free union.

If $s_{x}|_{Q_{\lambda}} \neq \id_{Q_{\lambda}}$ holds for any $\lambda \in \Lambda$ and any $x \in Q_{\lambda}$, then there is a canonical bijection 
\begin{equation*}
    \prod_{\lambda \in \Lambda} \good(Q_{\lambda},\ast_{\lambda}) \cong \good(\bigsqcup^{\mathrm{free}}_{\lambda \in \Lambda} Q_{\lambda}).
\end{equation*}
\end{lemma}
\begin{proof}
Let $\rho \in \good(\bigsqcup^{\mathrm{free}}_{\lambda \in \Lambda} Q_{\lambda})$ be a good involution.
By the definition of good involutions,
for any $x \in \bigsqcup^{\mathrm{free}}_{\lambda \in \Lambda} Q_{\lambda}$ and any $y \in Q_{\lambda}$, we have $x \ast_{\mathrm{free}} \rho(y) = x \ast_{\mathrm{free}}^{-1} y$, i.e.  $s_{\rho(y)} = s_{y}^{-1}$ as maps.
By the assumption that $s_{x}|_{Q_{\lambda}} \neq \id_{Q_{\lambda}}$, it follows that $s_{y}^{-1}|_{Q_{\lambda}} \neq \id_{Q_{\lambda}}$.

Assume on the contrary that $\rho(y) \notin Q_{\lambda}$ for some $y \in Q_{\lambda}$. Then by the definition of the interaction-free union, it follows that $s_{\rho(y)}|_{Q_{\lambda}} = \id_{Q_{\lambda}}$, which contradicts the assumption that $s_{y}^{-1}|_{Q_{\lambda}} \neq \id_{Q_{\lambda}}$. 
Hence, we obtain $\rho(y) \in Q_{\lambda}$, which implies $\rho(Q_{\lambda}) \subseteq Q_{\lambda}$.

Since $\rho$ is an involution, by applying $\rho$ to both sides of $\rho(Q_{\lambda}) \subseteq Q_{\lambda}$, we have $Q_{\lambda} = \rho^{2}(Q_{\lambda}) \subseteq \rho(Q_{\lambda})$. 
Thus, we obtain $\rho(Q_{\lambda}) = Q_{\lambda}$ for all $\lambda \in \Lambda$.

Therefore, the restriction satisfies $\rho|_{Q_{\lambda}} \in \good(Q_{\lambda})$ for each $\lambda \in \Lambda$, 
which induces the inclusion:
\begin{equation*}
    \good(\bigsqcup^{\mathrm{free}}_{\lambda \in \Lambda} Q_{\lambda}) 
    \hookrightarrow
    \prod_{\lambda \in \Lambda} \good(Q_{\lambda},\ast_{\lambda}).
\end{equation*}
Since this inclusion is the inverse map of the one constructed in Proposition~\ref{prop: inclusion of decomposition interaction-free union}, we obtain the desired bijection.
\end{proof}
%%%

\begin{thm} \label{main thm: decomp of good inv on interaction-free}
Let $\{(Q_{\lambda},\ast_{\lambda})\}_{\lambda \in \Lambda}$ be a family of nontrivial connected quandles and $\bigsqcup^{\mathrm{free}}_{\lambda \in \Lambda} Q_{\lambda}$ denote their interaction-free union.
Then there is a canonical bijection 
\begin{equation*}
    \good(\bigsqcup^{\mathrm{free}}_{\lambda \in \Lambda} Q_{\lambda})
    \cong
    \prod_{\lambda \in \Lambda} \good(Q_{\lambda},\ast_{\lambda}).
\end{equation*}
\end{thm}
\begin{proof}
We show that this interaction-free union $\bigsqcup^{\mathrm{free}}_{\lambda \in \Lambda} Q_{\lambda}$ satisfies the assumption of Lemma~\ref{lem: qdle operation is not identity implies decomposition}.

Assume on the contrary that there exists $\lambda \in \Lambda$ and $a \in Q_{\lambda}$ such that $s_{a}|_{Q_{\lambda}} = \id_{Q_{\lambda}}$.
Then by the definition of quandles, it follows that 
\begin{equation*}
    (x\ast_{\mathrm{free}}a)\ast_{\mathrm{free}} b 
    =
    (x \ast_{\mathrm{free}} b) \ast_{\mathrm{free}} (a \ast_{\mathrm{free}} b) 
\end{equation*}
for all $x,a,b \in Q_{\lambda}$. 
Since $s_{b} \circ s_{a}(x) = s_{a\ast_{\mathrm{free}} b} \circ s_{b} (x)$ holds for all $x \in Q_{\lambda}$,
we obtain the following identity of maps on $Q_{\lambda}$:
\begin{equation*}
    s_{b} \circ s_{a} = s_{a\ast_{\mathrm{free}} b} \circ s_{b},
\end{equation*}
from which it follows that 
\begin{equation*}
    s_{a\ast_{\mathrm{free}} b}
    =
    s_{b} \circ s_{a}\circ s_{b}^{-1}
    =
    s_{b} \circ \id_{Q_{\lambda}} \circ s_{b}^{-1}
    =
    \id_{Q_{\lambda}},
\end{equation*}
i.e. $s_{a\ast_{\mathrm{free}} b}|_{Q_{\lambda}} = \id_{Q_{\lambda}}$.

By the connectedness of $Q_{\lambda}$, it follows that for any $c \in Q_{\lambda}$ there exist $b_{1},\ldots,b_{n} \in Q_{\lambda}$ such that 
\begin{equation*}
    c = s_{b_{n}}^{\varepsilon_{n}} \circ \cdots \circ s_{b_{1}}^{\varepsilon_{1}}(a),
\end{equation*}
where $\varepsilon_{1},\ldots,\varepsilon_{n} \in \{\pm1\}$.
From the above argument, we have 
\begin{equation*}
    s_{c} 
    = 
    s_{s_{b_{n}}^{\varepsilon_{n}} \circ \cdots \circ s_{b_{1}}^{\varepsilon_{1}}(a)}
    =
    s_{b_{1}}^{\varepsilon_{1}}\circ \cdots \circ s_{b_{n}}^{\varepsilon_{n}}\circ s_{a} \circ s_{b_{n}}^{-\varepsilon_{n}} \circ \cdots \circ s_{b_{1}}^{-\varepsilon_{1}}.
\end{equation*}
Thus, if $s_{a}|_{Q_{\lambda}} = \id_{Q_{\lambda}}$, then
\begin{equation*}
    s_{c}|_{Q_{\lambda}} = s_{s_{b_{n}}^{\varepsilon_{n}} \circ \cdots \circ s_{b_{1}}^{\varepsilon_{1}}(a)}|_{Q_{\lambda}} = \id_{Q_{\lambda}}.
\end{equation*}
This means that $Q_{\lambda}$ is a trivial quandle, which contradicts our assumption that $Q_{\lambda}$ is a nontrivial connected quandle for all $\lambda \in \Lambda$.

Therefore, by Lemma~\ref{lem: qdle operation is not identity implies decomposition}, we obtain the desired bijection.
\end{proof}

%%%%%%%%%%%%%%%%%%%%%%%%%%%%%%%%%%%%%%%
\section{Generalized symplectic quandles} \label{sec: symplectic quandle}

%%%%%%%%%%%%%%%%%%%%%%%%%%%%%%%%%%%%%%%
\subsection{$R$-modules with an antisymmetric bilinear form} \label{subsec: module with antisymmetric bilinear form}

We provide some facts on $R$-modules with an antisymmetric bilinear form.

\begin{defin}
    Let $R$ be a commutative ring with identity and $M$ an $R$-module with an antisymmetric bilinear form $\langle, \rangle \colon M \times M \to R$ such that $\langle x,x \rangle = 0$ for all $x \in M$.
    Let $M^{\vee}$ denote the dual of $M$.
    \begin{itemize}
        \item An antisymmetric bilinear form $\langle, \rangle$ on $M$ is said to be \emph{nondegenerate} if the map $M \to M^{\vee}$ defined by $x \mapsto \langle x, -\rangle$ is injective.
        \item An antisymmetric bilinear form $\langle, \rangle$ on $M$ is said to be \emph{unimodular} if the map $M \to M^{\vee}$ defined by $x \mapsto \langle x, -\rangle$ is bijective.
    \end{itemize} 
\end{defin}

Note that if $R$ is not a field, then the nondegeneracy of antisymmetric bilinear forms does not necessarily imply unimodularity.

\begin{defin}
In the same setting as above,
the set $\rad(M)$ defined by 
\begin{equation*}
    \rad (M) = \{x \in M \mid \langle x,y \rangle = 0 \text{ for any } y \in M\}
\end{equation*}
is called the \emph{radical} of $M$.
\end{defin}

It is easy to see that the radical $\rad (M)$ of $M$ is a submodule of $M$.
The antisymmetric bilinear form on $M$ is nondegenerate if and only if $\rad (M) = \{0\}$.

\begin{defin}
For an $R$-module $M$, we define the \emph{torsion submodule} $T(M)$ of $M$ by 
\begin{equation*}
    T(M) = \{m \in M \mid rm = 0 \text{ for some } 0 \neq r \in R\}.
\end{equation*}
\end{defin}

The following facts are well-known.

\begin{prop} \label{prop: torsion module and radical}
    If $R$ is an integral domain, then the following hold:
    \begin{enumerate}
        \item The torsion submodule $T(M)$ of $M$ is a submodule.
        \item $T(M)$ is a subset of the radical $\rad(M)$.
    \end{enumerate}    
\end{prop}

%%%%%%%%%%%%%%%%%%%%%%%%%%%%%%%%%%%%%%%
\subsection{Generalized symplectic quandles} \label{subsec: symplectic quandles}

We introduce the notion of \emph{generalized symplectic quandles}, which generalizes the notion of symplectic quandles introduced by Navas and Nelson \cite{MR2493966}.
Symplectic quandles were primarily studied over free $R$-modules of finite rank. By introducing the notion of generalized symplectic quandles, we can also handle $R$-modules that are not necessarily free or finite rank.

\begin{defin}\label{def: generalized symplectic quandles}
    Let $R$ be a commutative ring with identity and $M$ be an $R$-module with an antisymmetric bilinear form $\langle, \rangle\colon M \times M \to R$ such that $\langle x,x \rangle = 0$ for all $x \in M$. 
    We define a binary operation $\ast \colon M\times M \to M$ by 
    \begin{equation*}
        x \ast y \coloneqq x + \langle x,y \rangle y.
    \end{equation*}
    The pair $(M,\ast)$ is called a \emph{generalized symplectic quandle over $R$}.
\end{defin}

\begin{remark}
    The dual quandle $(M,\ast^{-1})$ of a generalized symplectic quandle $(M,\ast)$ is defined by 
    \begin{equation*}
        x \ast^{-1} y = x - \langle x,y \rangle y.
    \end{equation*} 
\end{remark}

It is easy to see that if $M$ is a free $R$-module of finite rank, then a generalized symplectic quandle is nothing but a symplectic quandle.
One of the main advantages of generalized symplectic quandles is that they expand the framework to encompass broader classes of modules, such as torsion modules or quotient modules.

\begin{example}
    If an antisymmetric bilinear form $\langle , \rangle \colon M \times M \to R$ is a zero map, i.e. $\langle x,y \rangle = 0$ for every $x,y \in M$, then the generalized symplectic quandle $(M,\ast)$ is a trivial quandle on $M$.
\end{example}

\begin{defin}
    We say that a generalized symplectic quandle is \emph{nontrivial} if it is not a trivial quandle.
\end{defin}

Many properties that hold for symplectic quandles also hold for generalized symplectic quandles.
Here, we list some of these properties, noting that some require appropriate modifications of their proofs.

\begin{prop}[{\cite[Proposition 3.4]{NY2604}}] \label{prop: right-action is linear}
Let $(M,\ast)$ be a generalized symplectic quandle over a commutative ring $R$ with identity. 
For each $y \in M$, we define a map $s_{y}\colon M \to M$ by $s_{y}(x) \coloneqq x \ast y$.
Then $s_{y}$ is an $R$-linear invertible map.
\end{prop}

We have seen that if the characteristic of $R$ is 2, then a generalized symplectic quandle over $R$ is a kei (cf.~\cite[Proposition 1]{MR2493966}).
We also show that the converse is true when $R$ is an integral domain.

\begin{prop}[{cf.~\cite[Proposition 3.11]{NY2604}}]
Suppose that $R$ is an integral domain.
If the characteristic of $R$ is not 2 and a generalized symplectic quandle $(M,\ast)$ over $R$ is a kei, then $(M,\ast)$ is a trivial quandle.
\end{prop}
\begin{proof}
Since $(M,\ast)$ over $R$ is a kei, we have $x \ast y = x \ast^{-1} y$ for all $x,y \in M$, which implies $2\langle x,y \rangle y = 0$.
Since the characteristic of $R$ is not 2, we obtain $\langle x,y \rangle y = 0$.

Fix $y \in M$.
If $y \in T(M)$, then the inclusion $T(M) \subseteq \rad(M)$ implies $\langle x,y \rangle = 0$ for all $x \in M$.
If $y \notin T(M)$, then $\langle x,y \rangle y = 0$ implies $\langle x,y \rangle = 0$ because $M\setminus T(M)$ is torsion-free.

Thus, $\langle x,y \rangle = 0$ holds for all $x,y \in M$, i.e. $(M,\ast)$ is a trivial quandle.
\end{proof}

\begin{cor}[{cf.~\cite[Proposition 3.13]{NY2604}}]
Suppose that $R$ is an integral domain.
A nontrivial generalized symplectic quandle over $R$ is a kei if and only if the characteristic of $R$ is 2.
\end{cor}

\subsection{Decomposition of a generalized symplectic quandle as an interaction-free union} \label{subsec: symplectic quandles as interaction-free union}

Let $R$ be a commutative ring with identity and $M$ an $R$-module with an antisymmetric bilinear form $\langle, \rangle\colon M \times M \to R$.
In this section, we show that a generalized symplectic quandle $(M,\ast)$ decomposes as an interaction-free union of quandles.

The radical $\rad(M)$ of $M$ is a submodule of $M$, which forms a subquandle of $(M,\ast)$.
This subquandle $(\rad(M),\ast)$ is a trivial quandle (cf.~\cite[Proposition 2]{MR2493966}).

On the other hand, although the complement $M \setminus \rad (M)$ is not a submodule of $M$ in general unless $R$ is a field, it inherits a subquandle structure from $(M,\ast)$.

\begin{lemma}
Let $M$ be a $R$-module with an antisymmetric bilinear form $\langle, \rangle\colon M \times M \to R$.
If $M \neq \rad(M)$, then $M \setminus \rad (M)$ is a subquandle of $(M,\ast)$.
\end{lemma}
\begin{proof}
We first show that for every $x,y \in M \setminus \rad(M)$, $x \ast y \in M \setminus \rad(M)$ holds.

Assume on the contrary that for every $x \in M \setminus \rad(M)$, there exists $y \in M \setminus \rad(M)$ such that $x \ast y \in \rad(M)$.
Then
\begin{align*}
    0 &= \langle x \ast y, y \rangle \\
    &= \langle x +\langle x,y \rangle y,y \rangle \\
    &=
    \langle x,y \rangle + \langle x,y \rangle \langle y,y\rangle \\
    &=
    \langle x,y \rangle.
\end{align*}
Namely, we have $\langle x,y \rangle = 0$.
This implies that $x \ast y = x$, which contradicts the assumption that $x \in M \setminus \rad(M)$.

Thus, $x \ast y \in M \setminus \rad(M)$ for all $x,y \in M \setminus \rad(M)$.
Similarly, $x \ast^{-1} y \in M \setminus \rad(M)$ for all $x,y \in M \setminus \rad(M)$. 
Therefore, $M \setminus \rad(M)$ is a subquandle of $M$.
\end{proof}

\begin{prop} \label{prop: interaction-free symplectic quandles}
Let $M$ be a $R$-module with an antisymmetric bilinear form $\langle, \rangle\colon M \times M \to R$.
If $M \neq \rad(M)$, then the generalized symplectic quandle $(M,\ast)$ is an interaction-free union of $M \setminus \rad(M)$ and $\rad(M)$, i.e. 
\begin{equation*}
    M = (M \setminus \rad(M)) \bigsqcup^{\mathrm{free}} \rad(M).
\end{equation*}
\end{prop}
\begin{proof}
We have seen that both $M \setminus \rad(M)$ and $\rad(M)$ are subquandles of $(M,\ast)$.
We show that the operation $\ast$ coincides with the operation $\ast_{\mathrm{free}}$, where $\ast_{\mathrm{free}}$ denotes the operation on the interaction-free union $(M \setminus \rad(M)) \bigsqcup^{\mathrm{free}} \rad(M)$.

It is easy to see that $x \ast y = x \ast_{\mathrm{free}} y$ holds for all $x,y \in M \setminus \rad(M)$ or for all $x,y \in \rad(M)$.

For $x \in M \setminus \rad(M)$ and $y \in \rad(M)$, we have 
\begin{equation*}
    x \ast y 
    = x + \langle x,y \rangle y 
    = x 
    = x \ast_{\mathrm{free}} y
\end{equation*}
and 
\begin{equation*}
    y \ast x 
    = y + \langle y,x \rangle x 
    = y 
    = y \ast_{\mathrm{free}} x.
\end{equation*}
These imply that $x \ast y = x \ast_{\mathrm{free}} y$ for all $x,y \in M$.
Therefore, $(M,\ast)$ is an interaction-free union of $M \setminus \rad(M)$ and $\rad(M)$.
\end{proof}

%%%%%%%%%%%%%%%%%%%%%%%%%%%%%%%%%%%%%%%
\section{Good involutions on generalized symplectic quandles} \label{sec: good involutions on symplectic quandles}

In our previous work \cite{NY2604}, we studied the nonexistence of good involutions on symplectic quandles from the viewpoint of the properties of $R$ or $\langle, \rangle$.
In particular, when $R$ is an integral domain of characteristic 2, we showed the following:

\begin{thm}[{\cite[Theorem 5.3]{NY2604}}] \label{thm: main2 in NY}
Suppose that $R$ is an integral domain of characteristic 2. 
Let $(M,\ast)$ be a symplectic quandle over $R$.
If the antisymmetric bilinear form $\langle,\rangle\colon M \times M \to R$ is nondegenerate, then $\good(M,\ast) = \{\id_{M}\}$.
\end{thm}

When $R$ is an integral domain of characteristic 2, faithfulness of generalized symplectic quandles over $R$ can be characterized by the nondegeneracy of the bilinear form.

\begin{prop} \label{prop: faithfulness}
Let $R$ be an integral domain of characteristic 2 and $(M,\ast)$ a generalized symplectic quandle over $R$.
Then, $(M,\ast)$ is faithful if and only if the bilinear form $\langle, \rangle$ is nondegenerate.
\end{prop}
\begin{proof}
Suppose that the bilinear form $\langle, \rangle$ is nondegenerate.
If $s_x = s_y$ holds for $x,y \in M$, then for any $z \in M$, we have 
\begin{equation}\label{eq: faithfulness}
    \langle z,x \rangle x = \langle z,y \rangle y.
\end{equation}
We show that $x = y$.

First, if $x = 0$, then by Equation \eqref{eq: faithfulness}, we have $0 = \langle z,y \rangle y$, i.e. $y \in \rad(M)$ or $y \in T(M)$. By Proposition~\ref{prop: torsion module and radical}, it follows that $y \in \rad(M)$, which implies that $y = 0$ by the assumption that $\langle, \rangle$ is nondegenerate.
Thus, we have $x = y = 0$.

Next, if $x \neq 0$, then by the nondegeneracy of $\langle,\rangle$, there exists $z_0 \in M$ such that $\langle z_0,x \rangle \neq 0$. 
By Equation \eqref{eq: faithfulness}, we have 
\begin{equation*}
    \langle z_0,x \rangle x = \langle z_0,y \rangle y.
\end{equation*}
This implies that $\langle z_0,y \rangle \neq 0$.
Indeed, if $\langle z_0,y \rangle = 0$, then $\langle z_0,x \rangle x = 0$, i.e. $x = 0$, which contradicts our assumption that $x \neq 0$.
Thus, there exists $a,b\in R\setminus\{0\}$ such that $ax =by$.

By applying $b$ to both sides of $\langle z,x \rangle x = \langle z,y \rangle y$, it follows that 
\begin{equation*}
    b\langle z,x \rangle x 
    = b\langle z,y \rangle y
    =\langle z,y \rangle (by)
    =\langle z,y \rangle (ax),
\end{equation*}
i.e. $(b \langle z,x \rangle - a \langle z,y \rangle)x = 0$.
Since $x \notin \rad(M)$ implies $x \notin T(M)$ by Proposition~\ref{prop: torsion module and radical} and $x \neq 0$, 
it follows that $b \langle z,x \rangle = a \langle z,y \rangle$, 
i.e. $\langle z, bx - ay \rangle = 0$ holds for $z \in M$.
By the nondegeneracy of $\langle, \rangle$, it follows that $bx = ay$.

From $ax = by$, we have $abx - b^2 y = 0$. From $bx = ay$, it follows that
\begin{equation*}
    0 = abx - b^2y = a(ay) -b^2y = (a^2 - b^2)y, 
\end{equation*}
i.e. $(a^2 - b^2) y = 0$. Since $y \notin T(M)$ and $y \neq 0$, $a^2 = b^2$, i.e. $a = b$.

Since we have $ax = by$ and $a = b$, it follows that $a(x - y) = 0$.
Since $a \neq 0$, $x - y \in T(M)$. By the nondegeneracy of $\langle,\rangle$, it follows that $T(M) = \{0\}$. Thus we have $x - y = 0$, i.e. $x = y$.

Therefore, we show that $(M,\ast)$ is faithful.

Conversely, suppose that $\langle,\rangle$ is degenerate.
Then there exists $x \in \rad(M)$ such that $x \neq 0$.
For all $y \in M$, we have $y \ast x = y = y \ast 0$.
This implies that $s_{x} = s_{0}$, i.e. $(M,\ast)$ is not faithful.
Therefore, considering the contraposition, we show that if $(M,\ast)$ is faithful, then $\langle,\rangle$ is nondegenerate.
\end{proof}

\begin{remark}
Suppose that $M \neq \{0\}$ and the characteristic of $R$ is not $2$.
Then $(M,\ast)$ is not faithful even if $\langle, \rangle$ is nondegenerate.

Indeed,  
for any $z \in M$ and any $x \in M \setminus \{0\}$, we have $z \ast x = z + \langle z,x \rangle x = z + \langle z, -x \rangle (-x) = z \ast (-x)$, i.e. $s_{x} = s_{-x}$.
Since $x \neq -x$ for any $x \neq 0$, it follows that $(M,\ast)$ is not faithful.
\end{remark}

\begin{cor}
Let $R$ be an integral domain, $M$ a module over $R$ and $\langle,\rangle \colon M \times M \to R$ an antisymmetric bilinear form on $M$.
Then the induced bilinear form on the quotient module $M{/\rad(M)}$ is nondegenerate.
In particular, if the characteristic of $R$ is 2, then the generalized symplectic quandle $(M{/\rad(M)},\ast)$ is faithful, i.e. we obtain
\begin{equation*}
    \good(M{/\rad(M)},\ast) = \{\id\}.
\end{equation*} 
\end{cor}

%%%
Moreover, every good involution preserves the decomposition of generalized symplectic quandles:
\begin{prop}\label{prop: good involution preserves the decomposition of generalized symplectic quandles}
Let $R$ be an integral domain, $M$ a module over $R$ and $\langle, \rangle \colon M \times M \to R$ an antisymmetric bilinear form on $M$.
Then every good involution $\rho \colon M \to M$ on the generalized symplectic quandle $(M,\ast)$ satisfies $\rho(\rad(M)) = \rad(M)$.
\end{prop}
\begin{proof}
By the second condition of good involutions, we have $x \ast \rho(y) = x \ast^{-1} y$ for all $x,y \in M$.
Hence, we have 
\begin{equation*}
    \langle x, \rho(y) \rangle \rho(y) = -\langle x,y \rangle y.
\end{equation*}
Suppose $y \in \rad(M)$. If $\rho(y) = 0$, then it is obvious that $\rho(y) \in \rad(M)$. 
Thus, we may assume that $\rho(y) \neq 0$.
By the above equation, we have 
\begin{equation*}
    \langle x, \rho(y) \rangle \rho(y) = 0.
\end{equation*}
If $\langle x,\rho(y) \rangle = 0$ for any $x \in M$, then $\rho(y) \in \rad(M)$. 
Otherwise, since $\rho(y) \neq 0$, $\rho(y) \in T(M)$.
In this case, Proposition~\ref{prop: torsion module and radical} implies that $\rho(y) \in \rad(M)$.
Therefore, we obtain $\rho(\rad(M)) \subseteq \rad(M)$.

Since $\rho$ is an involution, applying $\rho$ to both sides gives 
\begin{equation*}
    \rad(M) = \rho^{2}(\rad(M)) \subseteq \rho(\rad(M)) \subseteq \rad(M),
\end{equation*}
which implies that $\rho(\rad(M)) = \rad(M)$.
\end{proof}

%%%%%%%%%%%
\subsection{Decomposition of good involutions on generalized symplectic quandles}

In this section, we consider good involutions on generalized symplectic quandles $(M,\ast)$ over an integral domain $R$ of arbitrary characteristic equipped with a general antisymmetric bilinear form$\langle, \rangle$.
Our key idea is to regard a generalized symplectic quandle as an interaction-free union of quandles, which allows us to apply the results on good involutions to this setting.

First, we show that if $R$ is an integral domain, then $\rad(M)$ is the maximal trivial component of $(M,\ast)$.

\begin{lemma}\label{lem: radical is maximal trivial component}
Let $R$ be an integral domain and $(M,\ast)$ a generalized symplectic quandle over $R$.
Then the radical $\rad(M)$ is the maximal trivial component of $(M,\ast)$.
\end{lemma}
\begin{proof}
Let $D$ denote the maximal trivial component of $(M,\ast)$.
By the definition, $\rad(M) \subseteq D$.
We show the other inclusion.

For all $x \in D$ and $y \in M$, we have $y \ast x = y$.
Since $y \ast x = y + \langle y,x\rangle x$, it follows that 
\begin{equation*}
    \langle y,x\rangle x = 0.
\end{equation*}
If $\langle y,x\rangle = 0$ for all $y \in M$, then $x \in \rad(M)$. If there exists $y \in M$ such that $\langle y,x\rangle \neq 0$, then $x \in T(M)$.
By Proposition~\ref{prop: torsion module and radical}, $x \in \rad(M)$ in either cases, which shows that $D \subseteq \rad(M)$.

Therefore, we obtain $D = \rad(M)$.
\end{proof}

\begin{thm} \label{main thm: symplectic qdle}
Let $R$ be an integral domain and $(M,\ast)$ a generalized symplectic quandle over $R$.
If $M \neq \rad(M)$, then there is a canonical bijection
\begin{equation*}
    \good(M,\ast) \cong \good(M\setminus \rad(M),\ast) \times \mathrm{Inv}(\rad(M)).
\end{equation*} 
\end{thm}
\begin{proof}
Since the radical $\rad(M)$ of $M$ is a trivial quandle, we have $\good(\rad(M),\ast) = \mathrm{Inv}(\rad(M))$.
Moreover, by Proposition~\ref{prop: interaction-free symplectic quandles}, the generalized symplectic quandle $(M,\ast)$ is an interaction-free union of the subquandles $(M \setminus \rad(M),\ast), (\rad(M),\ast)$.
Furthermore, if $R$ is an integral domain, by Lemma~\ref{lem: radical is maximal trivial component}, the radical $\rad(M)$ is the maximal trivial component of $(M,\ast)$.

Since $M \neq \rad(M)$ and $0 \in \rad(M)$, it follows that $\rad(M) \neq \emptyset$ and $M \setminus \rad(M) \neq \emptyset$.
Therefore, by Theorem~\ref{main thm: decomp of maximal trivial component}, we obtain the desired bijection.
\end{proof}

Since the radical $\rad(M)$ of $M$ always contains the zero in $M$, we have $\mathrm{Inv}(\rad(M)) \neq \emptyset$.
Thus we obtain the following:

\begin{cor}\label{cor: emptyness of good inv on symplectic quandles}
Let $R$ be an integral domain and $(M,\ast)$ a generalized symplectic quandle over $R$.
Then, $\good(M,\ast) = \emptyset$ if and only if $\good(M\setminus \rad(M),\ast) = \emptyset$.
\end{cor}

%%%%%%%%%%%
\subsection{When $R$ is an integral domain of characteristic 2}

When we do not assume that the bilinear form $\langle,\rangle$ is nondegenerate, we obtain the following:

\begin{lemma}\label{lem: good inv is expansion char2}
Let $R$ be an integral domain of characteristic 2 and $(M,\ast)$ a generalized symplectic quandle over $R$.
If $M \neq \rad(M)$, then for any $y \in M \setminus \rad(M)$ and any $\rho \in \good(M\setminus \rad(M))$, there exists $\lambda \in \mathrm{Frac}(R)$ such that $\rho(y) = \lambda y$ in $M \otimes_{R}\mathrm{Frac}(R)$.
\end{lemma}
\begin{proof}
Since $y \notin \rad(M)$, there exists $x \in M \setminus \rad(M)$ such that $\langle x,y \rangle \neq 0$.
Thus, in the same way as \cite[Lemma 5.2]{NY2604}, there exists $\lambda \in \mathrm{Frac}(R)$ such that $\rho(y) = \lambda y$ in $M \otimes_{R}\mathrm{Frac}(R)$.
\end{proof}

\begin{prop}\label{prop: good inv is identity cha2}
Let $R$ be an integral domain of characteristic 2 and $(M,\ast)$ a generalized symplectic quandle over $R$.
If the underlying module $M$ is torsion-free and $M \neq \rad(M)$, then for any $\rho \in \good(M)$,
we have $\rho|_{M\setminus \rad(M)} = \id_{M\setminus \rad(M)}$.
\end{prop}
\begin{proof}
By Lemma~\ref{lem: good inv is expansion char2}, for any $y \in M \setminus \rad(M)$, there exists $\lambda \in \mathrm{Frac}(R)$ such that $\rho(y) = \lambda y$ in $M \otimes_{R}\mathrm{Frac}(R)$.

Since $\rho \in \good(M)$, for any $x \in M$, we have $x \ast \rho(y) = x \ast^{-1} y$, i.e.
\begin{equation*}
    \langle x, \rho(y) \rangle \rho(y) = \langle x,y \rangle y.
\end{equation*}
Since $\rho(y) = \lambda y$, we obtain $(\lambda^2 +1)\langle x,y \rangle y =0$.

$y \notin \rad(M)$ implies that there exists $x \in M \setminus \rad(M)$ such that $\langle x,y \rangle \neq 0$.
Moreover, $y \notin T(M)$ implies $\lambda^2 + 1 = 0$, i.e. $\lambda = 1$.
Therefore, for any $y \in M\setminus \rad(M)$,
we have $\rho(y) = y$.
\end{proof}

\begin{remark}
Since we do not suppose that the bilinear form $\langle, \rangle$ is nondegenerate, we can only say that $\rho(0) \in T(M)\subseteq \rad(M)$ for $\rho \in \good(M)$, i.e. $\rho(0) = 0$ does not hold in general (cf. \cite[Lemma 5.1]{NY2604}).
\end{remark}

\begin{remark}
It is essential in Proposition~\ref{prop: good inv is identity cha2} to assume that the underlying module $M$ is torsion-free. 
Without this assumption, one can only conclude that $\rho(y)-y \in T(M)$ for every $y \in M \setminus \rad(M)$ (see Example~\ref{eg: with torsion-free}).
\end{remark}

From the discussions above, we have a generalization of \cite[Theorem 5.3]{NY2604} as follows:

\begin{cor}\label{cor: char2 implies good is inv of rad}
Let $R$ be an integral domain of characteristic 2 and $(M,\ast)$ a generalized symplectic quandle over $R$.
If the underlying module $M$ is torsion-free, then there exists a bijection
\begin{equation*}
    \good(M,\ast) \cong \operatorname{Inv}(\rad(M)).
\end{equation*}
In particular, if the bilinear form $\langle, \rangle$ is nondegenerate, then we have $\good(M,\ast) = \{\id_{M}\}$.
\end{cor}
\begin{proof}
If $M = \rad(M)$, then we have 
\begin{equation*}
    \good(M,\ast) = \good(\rad(M),\ast) = \operatorname{Inv}(\rad(M)).
\end{equation*}
If $M \neq \rad(M)$, then by Theorem \ref{main thm: symplectic qdle} and Proposition \ref{prop: good inv is identity cha2} we have 
\begin{align*}
    \good(M,\ast) 
    &\cong \good(M\setminus \rad(M),\ast) \times \operatorname{Inv}(\rad(M)) \\
    &= \{\id_{M\setminus \rad(M)}\} \times \operatorname{Inv}(\rad(M)) \\
    &\cong \operatorname{Inv}(\rad(M)).
\end{align*}
This completes the proof.
\end{proof}

This bijection implies that the set $\good(M,\ast)$ of good involutions on a generalized symplectic quandle $(M,\ast)$ reflects the degeneracy of the underlying torsion-free module $M$.

Since the number of involutions on a finite set is given by the \emph{telephone numbers}, we obtain the following corollary:

\begin{cor}
Let $\mathbb{F}_{2}$ denote the finite field of order two. 
Let $(M,\ast)$ be a generalized symplectic quandle over $\mathbb{F}_{2}$.
If $n \coloneqq \lvert \rad(M) \rvert$, then 
\begin{equation*}
    \lvert \good(M,\ast) \rvert
    =
    \sum_{k=0}^{\lfloor n/2 \rfloor} \frac{n !}{2^{k} k! (n-2k)!}.
\end{equation*}
\end{cor}

We end this section by providing an example of a good involution on generalized symplectic quandles when the underlying module has torsion.

\begin{example}\label{eg: with torsion-free}
Let $R = \mathbb{F}_{2}[t]$ be the polynomial ring over $\mathbb{F}_{2}$.
We define an $R$-module $M$ by 
\begin{equation*}
    M \coloneqq Re_{1}\oplus Re_{2} \oplus (R/(t))u,
\end{equation*}
where $e_{1},e_{2}$ are free generators and $u$ denotes a generator of the quotient module $R/(t)$ (i.e. $u = 1 + (t)$).
Note that $tu = 0$ in $M$ since 
\begin{equation*}
    tu = t(1+(t)) = t + (t) = 0.
\end{equation*}

We define an antisymmetric bilinear form $\langle,\rangle \colon M \times M \to R$ by 
\begin{equation*}
    \langle ae_{1} + be_{2} + cu, a^{\prime}e_{1} + b^{\prime}e_{2} + c^{\prime}u \rangle 
    \coloneqq
    t(ab^{\prime} + a^{\prime}b).
\end{equation*}
Then we have $\rad(M) = (R/(t))u$.

Consider the map $\rho \colon M \to M$ defined by $\rho(x) \coloneqq x + u$.
We claim that $\rho$ is a good involution on the generalized symplectic quandle $(M,\ast)$ satisfying $\rho|_{M \setminus \rad(M)} \neq \id$.

Since $\rho^2(x) = \rho(x+u) = (x+u)+u = x$ for all $x \in M$, the map $\rho$ is an involution.

First, we show that $\rho(x \ast y) = \rho(x) \ast y$ for all $x,y \in M$.
Since $u \in \rad(M)$, we have 
\begin{align*}
    \rho(x\ast y) 
    &= \rho(x + \langle x,y \rangle y)\\
    &= x + \langle x,y \rangle y + u \\
    &= x + u + \langle x+u,y \rangle y\\
    &= \rho(x) + \langle \rho(x),y \rangle y \\
    &= \rho(x) \ast y.
\end{align*}
Thus, we obtain $\rho(x \ast y) = \rho(x) \ast y$.

Next, we show that $x \ast \rho(y) = x \ast^{-1} y$ for all $x,y \in M$.
Since $tu = 0$ and $\langle x,y \rangle \in tR$ for all $x,y \in M$, we have 
\begin{align*}
    x \ast \rho(y)
    &= x + \langle x,\rho(y) \rangle \rho(y) \\
    &= x + \langle x, y+u \rangle (y+u) \\
    &= x + \langle x,y \rangle y + \langle x,y \rangle u \\
    &= x + \langle x,y \rangle y \\
    &= x \ast y.
\end{align*}
Thus, we obtain $x \ast \rho(y) = x \ast y = x \ast^{-1} y$.

Therefore, $\rho$ is a good involution on the generalized symplectic quandle $(M,\ast)$.
\end{example}

%%%%%%%%%%%%%%%%%%%%%%%%%%%%%%%%%%%%%%%
\section{Good involutions on a direct product of quandles} \label{sec: good involutions on product}

Given a family of quandles, one can construct a quandle structure on their Cartesian product; this construction is called the \emph{direct product of quandles}.

In this section, we investigate the relationship between good involutions on a direct product of quandles and those on each component.
As a consequence, we provide a counterexample to a conjecture proposed by Ta \cite[Conjecture 11.8]{ta2025goodinvolutionsconjugationsubquandles}.

%%%%%%%%%%%%%%%%%%%%%%%%%%%%%%%%%%%%%%%
\subsection{Decomposition of good involutions on a direct product of quandles}
%%%
\begin{defin}\label{def: direct product}
Let $(Q_{1},\ast_{1}),(Q_{2},\ast_{2})$ be quandles. 
We define a binary operation $\ast$ on the direct product $Q_{1} \times Q_{2}$ by 
\begin{equation*}
    (x_{1},x_{2}) \ast (y_{1},y_{2}) 
    \coloneqq 
    (x_{1}\ast_{1}y_{1}, x_{2}\ast_{2}y_{2}).
\end{equation*}
The pair $(Q_{1}\times Q_{2}, \ast)$ is a quandle, which is called a \emph{direct product of quandles}.
\end{defin}
%%%

In what follows, we compare the set of good involutions on a direct product of quandles with the product of those on its components.
%%%
\begin{prop}\label{prop: injective good involutions on direct products}
Let $\{(Q_{\lambda},\ast_{\lambda})\}_{\lambda \in \Lambda}$ be a given family of quandles
and $(\prod_{\lambda \in \Lambda}Q_{\lambda},\ast)$ denote their direct product.
Then there exists a canonical inclusion
\begin{equation*}
    \prod_{\lambda \in \Lambda} \good(Q_{\lambda},\ast_{\lambda})
    \hookrightarrow
    \good(\prod_{\lambda \in \Lambda}Q_{\lambda}).
\end{equation*}
\end{prop}
\begin{proof}
We define a map $G\colon \prod_{\lambda \in \Lambda} \good(Q_{\lambda},\ast_{\lambda}) \to \good(\prod_{\lambda \in \Lambda}Q_{\lambda})$ by 
\begin{equation*}
    G((\rho_{\lambda})_{\lambda \in \Lambda})((x_{\lambda})_{\lambda \in \Lambda}) \coloneqq (\rho_{\lambda}(x_{\lambda}))_{\lambda \in \Lambda}.
\end{equation*}
It is straightforward to verify that $G((\rho_{\lambda})_{\lambda \in \Lambda})$ is indeed a good involution on the direct product $\prod_{\lambda \in \Lambda} Q_{\lambda}$. Thus, $G$ is well-defined and injective.
\end{proof}

The following example shows that the complete decomposition does not hold in general.

%%%
\begin{example}\label{eg: family of trivial qdle, product}
Let $A = \{a_{1},a_{2}\},B = \{b_{1},b_{2}\}$ be sets of two elements and $A_{\mathrm{tri}},B_{\mathrm{tri}}$ denote the trivial quandles on $A,B$, respectively.

Since the direct product of these quandles satisfies
\begin{equation*}
    A_{\mathrm{tri}}\times B_{\mathrm{tri}}
    \cong 
    (A\times B)_{\mathrm{tri}},
\end{equation*}
we have 
\begin{equation*}
    \good(A_{\mathrm{tri}}\times B_{\mathrm{tri}})
    =
    \good((A\times B)_{\mathrm{tri}})
    =
    \operatorname{Inv}(A \times B).
\end{equation*}
On the other hand, it is well-known that $\operatorname{Inv}(A) \times \operatorname{Inv}(B) \subsetneq \operatorname{Inv}(A \times B)$ for sets of more than one element.
Since $\good(A_{\mathrm{tri}}) = \operatorname{Inv}(A)$ and $\good(B_{\mathrm{tri}}) = \operatorname{Inv}(B)$, we obtain 
\begin{equation*}
    \good(A_{\mathrm{tri}}) \times \good(B_{\mathrm{tri}})
    \subsetneq 
    \good(A_{\mathrm{tri}}\times B_{\mathrm{tri}}).
\end{equation*}
\end{example}

Next, we provide a sufficient condition for the equality to hold.
%%%
\begin{thm}\label{main thm: good involutions on direct products}
Let $\{(Q_{\lambda},\ast_{\lambda})\}_{\lambda = 1,\ldots,n}$ be a given finite family of nontrivial connected quandles
and $(\prod_{\lambda = 1}^{n}Q_{\lambda},\ast)$ denote their direct product.
Then there is a canonical bijection
\begin{equation*}
    \good(\prod_{\lambda = 1}^{n}Q_{\lambda})
    \cong 
    \prod_{\lambda = 1}^{n} \good(Q_{\lambda},\ast_{\lambda}).
\end{equation*}
\end{thm}
\begin{proof}
By Proposition~\ref{prop: injective good involutions on direct products}, we have a well-defined injection 
\begin{equation*}
    G\colon \prod_{\lambda=1}^{n} \good(Q_{\lambda},\ast_{\lambda}) \to \good(\prod_{\lambda=1}^{n}Q_{\lambda})
\end{equation*} by 
\begin{equation*}
    G((\rho_{\lambda})_{\lambda \in \Lambda})((x_{\lambda})_{\lambda \in \Lambda}) \coloneqq (\rho_{\lambda}(x_{\lambda}))_{\lambda \in \Lambda}.
\end{equation*}

To establish the surjectivity of $G$, it suffices to show that for any $\rho \in \operatorname{Good}(\prod_{\lambda=1}^{n}Q_{\lambda})$, the $i$-th component $\rho(x)_i$ depends only on $x_i$.
Precisely, we claim that if $x = (x_j)_j, x' = (x'_j)_j \in \prod_{\lambda=1}^{n}Q_{\lambda}$ satisfy $x_i = x'_i$, then $\rho(x)_i = \rho(x')_i$ for all $1 \leq i \leq n$.

Fix $i \in \{1, \ldots, n\}$.
Since each $Q_{\lambda}$ is connected, there exist elements $y_1, \ldots, y_m \in \prod_{\lambda=1}^{n}Q_{\lambda}$ and signs $\varepsilon_1, \ldots, \varepsilon_m \in \{\pm 1\}$ such that
\begin{equation}\label{eq: connected finite family}
    x' = s_{y_{m}}^{\varepsilon_{m}}\circ \cdots \circ s_{y_{1}}^{\varepsilon_{1}} (x).
\end{equation}
Since $x_i = x'_i$, we may choose the $i$-th component of each $y_k$ to be $(y_k)_i = x_i$ without affecting the connectivity in the other components.

By the definition of good involutions, we have $\rho \circ s_{y} = s_{y} \circ \rho$ for any good involution $\rho$ and every $y$, 
i.e. $\rho$ commutes with the point symmetry $s_{y}$ and its inverse $s_{y}^{-1}$.
Applying $\rho$ to Equation~\eqref{eq: connected finite family} yields 
\begin{equation}\label{eq: connected finite family and good involution}
\rho(x^{\prime}) = s_{y_{m}}^{\varepsilon_{m}}\circ \cdots \circ s_{y_{1}}^{\varepsilon_{1}} (\rho(x)).
\end{equation}

Moreover, by the definition of good involutions, we have $\rho(x) = \rho(x)\ast x$. Taking the $i$-th component, we have $\rho(x)_{i} = \rho(x)_{i} \ast_{i} x_{i}$, i.e.
\begin{equation*}
    s_{x_{i}}(\rho(x)_{i}) = \rho(x)_{i} \text{ and } s_{x_{i}}^{-1}(\rho(x)_{i}) = \rho(x)_{i}.
\end{equation*}

Since $x_{i} = x^{\prime}_{i}$, we may assume that $(y_{1})_{i} = \cdots = (y_{m})_{i} = x_{i}$.
Thus, Equation~\eqref{eq: connected finite family and good involution} implies that $\rho(x^{\prime})_{i} = \rho(x)_{i}$.
Therefore, the map $\rho_{i}\colon Q_{i} \to Q_{i}$ defined by $a \mapsto \rho(x)_{i}$, where $x \in \prod_{\lambda=1}^{n}Q_{\lambda}$ with $x_{i} = a$ is well-defined.

Componentwise evaluation shows that each $\rho_i$ is indeed a good involution on $Q_i$, and thus $G((\rho_1, \ldots, \rho_n)) = \rho$.
Therefore, $G$ is surjective, which completes the proof.
\end{proof}
%%%

\subsection{Connected noninvolutory symmetric quandles}
\label{subsec: cns quandles}

We apply the fact in the previous section to provide a counterexample to the following conjecture:
%%%
\begin{conj}[{\cite[Conjecture 11.8]{ta2025goodinvolutionsconjugationsubquandles}}] \label{conj: ta}
Let $(Q,\ast)$ be a connected noninvolutory quandle. 
Then we have $\lvert \good(Q,\ast) \rvert \leq 1$.
\end{conj}
%%%

Since the binary operation on the direct product of quandles is component-wise, we obtain the following facts:

%%%
\begin{lemma}\label{lem: product does not preserve involutory}
Let $(Q_{1},\ast_{1}), (Q_{2},\ast_{2})$ be quandles.
If $(Q_{1},\ast_{1})$ is an involutory quandle and $(Q_{2},\ast_{2})$ is a noninvolutory quandle, then the direct product $(Q_{1} \times Q_{2},\ast)$ is a noninvolutory quandle.
\end{lemma}
%%%
\begin{lemma}\label{lem: product preserve connectedness}
Let $(Q_{1},\ast_{1}), (Q_{2},\ast_{2})$ be quandles.
If $(Q_{1},\ast_{1})$ and $(Q_{2},\ast_{2})$ are connected, then the direct product $(Q_{1} \times Q_{2},\ast)$ is connected.
\end{lemma}
%%%

From these facts, we say that if $(Q_{1},\ast_{1})$ is a connected involutory quandle and $(Q_{2},\ast_{2})$ is a connected noninvolutory quandle such that $\lvert \good(Q_{1},\ast_{1}) \rvert > 1$ and $\good(Q_{2},\ast_{2}) \neq \emptyset$, then the direct product $(Q_{1} \times Q_{2},\ast)$ is a connected noninvolutory quandle such that $\lvert \good(Q_{1}\times Q_{2},\ast) \rvert > 1$, which will be a counterexample to Conjecture~\ref{conj: ta}. Indeed, we have an example of such quandles.

\begin{example}\label{eg: counterexample to conjecture by Ta}
Let $GL(2,\mathbb{F}_{3})$ be the general linear group over the finite field $\mathbb{F}_{3}$ of order three,
and let $A_{5}$ be the alternating group of degree $5$.
We define a conjugacy class $X_{1}$ in $GL(2,\mathbb{F}_{3})$ by 
\begin{equation*}
    X_{1} \coloneqq \{
        g^{-1}J g \mid g \in GL(2,\mathbb{F}_{3})
    \},
    \quad
    \text{where }
    J = \begin{bmatrix}
        0 & 1 \\
        1 & 0
    \end{bmatrix},
\end{equation*}
and let $X_{2}$ denote the conjugacy class of $3$-cycles in $A_{5}$.

Then we claim that the direct product $\operatorname{Conj}(X_{1}) \times \operatorname{Conj}(X_{2})$ of their conjugation quandles is a connected noninvolutory quandle equipped with multiple good involutions.

First, we show that $\operatorname{Conj}(X_{1})$ is a connected involutory quandle equipped with multiple good involutions.

Since $\det J = -1$, it follows that the subgroup $\langle X_{1} \rangle$ generated by $X_{1}$ is a normal subgroup of $GL(2,\mathbb{F}_{3})$ containing an element whose determinant is $-1$.
This implies that $\langle X_{1} \rangle =GL(2,\mathbb{F}_{3})$, i.e. $\operatorname{Conj}(X_{1})$ is connected.

Let $I$ denote the identity matrix.
Since $J \in X_{1}$ satisfies $J^2 = I$, i.e. the order of $J$ is $2$, it follows that the order of every element in $X_{1}$ is $2$.
Thus $\operatorname{Conj}(X_{1})$ is involutory, which implies that the identity map $\id_{X_{1}}$ is a good involution on $\operatorname{Conj}(X_{1})$.
Moreover, since $-I$ is in the center $Z(GL(2,\mathbb{F}_{3}))$ and $(-I)^{2} = I$, we obtain $-J \in X_{1}$.
Indeed, let $D \in GL(2,\mathbb{F}_{3})$ denote 
\begin{equation*}
    D = 
    \begin{bmatrix}
        1 & 0 \\
        0 & -1
    \end{bmatrix}.
\end{equation*}
Then we have 
\begin{equation*}
    D^{-1}JD 
    =
    \begin{bmatrix}
        1 & 0 \\
        0 & -1
    \end{bmatrix} 
    \begin{bmatrix}
        0 & 1 \\
        1 & 0
    \end{bmatrix}
    \begin{bmatrix}
        1 & 0 \\
        0 & -1 
    \end{bmatrix}
    =
    \begin{bmatrix}
        0 & 1 \\
        -1 & 0
    \end{bmatrix}
    \begin{bmatrix}
        1 & 0 \\
        0 & -1
    \end{bmatrix}
    = 
    -J,
\end{equation*}
i.e. $-J \in X_{1}$.
Since we have 
\begin{equation*}
    -I (g^{-1}Jg)
    =
    g^{-1}(-J)g 
    =
    g^{-1}(D^{-1}JD)g 
    =
    (Dg)^{-1}J (Dg),
\end{equation*}
it follows that $(-I)X_{1} = X_{1}$.
Thus, we can define an involution $\tau \colon X_{1} \to X_{1}$ by $\tau(u) = -u$.

Recall that $\operatorname{Conj}(X_{1}) = (X_{1},\ast)$ is involutory, i.e. $\ast = \ast^{-1}$.
For any $x,y \in X_{1}$, we have 
\begin{equation*}
    \tau(x \ast y) = -(x \ast y) = - y^{-1}xy = y^{-1}(-x)y 
    = (-x)\ast y = \tau (x)\ast y,
\end{equation*}
and 
\begin{equation*}
    x \ast \tau (y) = x \ast (-y) = (-y)^{-1}x (-y)
    = y^{-1}xy = x \ast y = x \ast^{-1} y.
\end{equation*}
Hence, the involution $\tau$ is a good involution on $\operatorname{Conj}(X_{1})$,
i.e. we obtain 
\begin{equation*}
    \{\id_{X_{1}}, \tau \} \subseteq \good(\operatorname{Conj}(X_{1})).
\end{equation*}
Thus, the conjugation quandle $\operatorname{Conj}(X_{1})$ is a connected involutory quandle satisfying $\lvert \good(\operatorname{Conj}(X_{1})) \rvert > 1$.

Next, we show that $\operatorname{Conj}(X_{2})$ is a connected noninvolutory symmetric quandle.

Since $A_{5}$ is generated by $3$-cycles, $\langle X_{2} \rangle = A_{5}$, i.e. $\operatorname{Conj}(X_{2})$ is connected.
Since every $3$-cyle of $A_{5}$ is order $3$, it follows that the conjugation quandle $\operatorname{Conj}(X_{2})$ is noninvolutory.
By the fact on good involutions on conjugation quandles, the map $\iota\colon X_{2} \to X_{2}$ defined by $\iota(\sigma) \coloneqq \sigma^{-1}$ is a good involution on $\operatorname{Conj}(X_{2})$, i.e. we obtain 
\begin{equation*}
    \{\iota\} \subseteq \good(\operatorname{Conj}(X_{2})).
\end{equation*}
Thus, the conjugation quandle $\operatorname{Conj}(X_{2})$ is a connected noninvolutory symmetric quandle.

By Lemma~\ref{lem: product does not preserve involutory} and Lemma~\ref{lem: product preserve connectedness}, the direct product $\operatorname{Conj}(X_{1}) \times \operatorname{Conj}(X_{2})$ is a connected noninvolutory quandle.
Furthermore, by Theorem~\ref{main thm: good involutions on direct products}, it follows that 
\begin{align*}
    &\lvert \good(\operatorname{Conj}(X_{1})) \times \good(\operatorname{Conj}(X_{2})) \rvert \\
    &=
    \lvert \good(\operatorname{Conj}(X_{1})) \rvert 
    \lvert \good(\operatorname{Conj}(X_{2})) \rvert \\
    &\geq 2\cdot 1 \\ &= 2.
\end{align*}
Therefore, the direct product $\operatorname{Conj}(X_{1}) \times \operatorname{Conj}(X_{2})$ has multiple good involutions.
\end{example}

Since $X_{1} \times X_{2}$ is a conjugacy class in the direct product group $GL(2,\mathbb{F}_{3}) \times A_{5}$, the product $\operatorname{Conj}(X_{1}) \times \operatorname{Conj}(X_{2})$ forms a conjugation subquandle.
Therefore, it serves as a counterexample to \cite[Conjecture 11.9]{ta2025goodinvolutionsconjugationsubquandles}.

One might consider imposing irreducibility (or \emph{direct indecomposablity}) as an additional condition in Conjecture~\ref{conj: ta}. 
However, counterexamples still exist even in the irreducible case, and these will be presented in our forthcoming paper.

%%%%%%%%%%%%%%%%%%%%%%%%%%%%%%%%%%%%%%%
\subsection*{Acknowledgements}
The authors are grateful to Takefumi Nosaka for helpful discussions, which inspired this work.

%%%%%%%%%%%%%%%%%%%%%%%%%%%%%%%%%%%%%%%
\bibliographystyle{alpha}
\bibliography{decomp}
%%%%%%%%%%%%%%%%%%%%%%%%%%%%%%%%%%%%%%%
\end{document}